\documentclass[12pt,a4paper]{amsart}
\usepackage{amsfonts}
\usepackage{amssymb}
\usepackage{amsthm}
\usepackage[centertags]{amsmath}
\usepackage{newlfont}
\usepackage{graphicx}

\newtheorem{theorem}{Theorem}
\newtheorem{corollary}[theorem]{Corollary}

\newtheorem{remark}[theorem]{Remark}

\def\r{\mathbb R}

\def\t{\textbf{t}}
\def\n{\textbf{n}}
\def\b{\textbf{b}}
\def\c{\textbf{c}}
\begin{document}
\title{Minimal   surfaces in Euclidean space with a log-linear density }
\author{Rafael L\'opez}
 \address{Departamento de Geometr\'{\i}a y Topolog\'{\i}a\\
Universidad de Granada\\
18071 Granada, Spain\\}
\thanks{Partially supported by MEC-FEDER
 grant no. MTM2011-22547 and Junta de Andaluc\'{\i}a grant no. P09-FQM-5088.}
 \email{ rcamino@ugr.es}

\begin{abstract}We study   surfaces in Euclidean space $\r^3$ that are minimal for a  log-linear density $\phi(x,y,z)=\alpha x+\beta y+\gamma y$, where $\alpha,\beta,\gamma$ are real numbers not all zero. We prove that if a surface is $\phi$-minimal foliated by circles in parallel planes, then these planes are orthogonal to the vector $(\alpha,\beta,\gamma)$ and  the surface must be rotational. We also classify all minimal surfaces of translation type.
\end{abstract}

 \maketitle

\section{Introduction and statement of results}

In the last years the study and the interest for manifolds with density has increased   due to its applications in probability and statistics. The literature on manifolds with density has increased so we only refer the introductory survey of Morgan in \cite{mo0}; see also  \cite{mo1,mo2}. In this paper we consider surfaces in  Euclidean space $\r^3$ with a positive smooth density function $e^\phi$ which is used to weight the volume and the surface area. The mean curvature $H_\phi$  of an oriented surface $S$ in $\r^3$ with density $e^\phi$ is
\begin{equation}\label{hg}
H_\phi=H-\frac12\frac{d\phi}{dN},
\end{equation}
where $N$ and $H$ stand for the Gauss map of $S$ and  its mean curvature, respectively. The expression \eqref{hg} for  $H_\phi$ is obtained  by the first variation of weighted area \cite{gr}. We say that $S$ is a constant $\phi$-mean curvature surface if $H_\phi$ is a constant function on $S$, and if $H_\phi=0$ everywhere, we say that $S$ is a $\phi$-minimal surface. In this context, one can pose similar problems as in the classical theory of minimal surfaces.

Among the many choices of densities, we focus in the simplest function $\phi$, that is,  $\phi$ is a linear function in its variables, and  we say then that  $e^\phi$ is a \emph{log-linear density}. We suppose that $\phi(x,y,z)=\alpha x+\beta y+\gamma z$, where $(x,y,z)$ are the canonical coordinates of $\r^3$ and $\alpha,\beta,\gamma$ are real number not all zero. We call the vector $\vec{v}=(\alpha,\beta,\gamma)$ the \emph{$\phi$-density vector}. A computation of \eqref{hg} implies that $H_\phi=0$ if and only if
\begin{equation}\label{hg2}
H=\frac12\langle N,\vec{v}\rangle.
\end{equation}
The interest for this type of density functions  appears in the singularity theory of the mean curvature flow, where    the constant $\phi$-mean curvature equation $H_\phi=ct$  is  the  equation of the limit flow by a proper blow-up procedure near type II singular points (for example, see \cite{hs,wa}). In the literature, a solution of this equation is also called  a translating soliton, or simply translator \cite{wh} . Thus, a translator with velocity $\vec{v}$ is equivalent to be a $\phi$-minimal surface with $\phi$-density vector $\vec{v}$ (\cite{il}).

On the other hand, and previously, Eq. \eqref{hg2} had already studied in the theory of PDE of elliptic type.  In a nonparametric form, and for the particular case that $\phi(x,y,z)=\gamma z$,  the graph of a function $u=u(x,y)$ defined on a domain $\Omega$ of the $xy$-plane satisfies \eqref{hg2} if and only if \begin{equation}\label{eqdiv}
\mbox{div}\left(\frac{\nabla u}{\sqrt{1+|\nabla u|^2}}\right)=\frac{\gamma}{\sqrt{1+|\nabla u|^2}}.
\end{equation}
This equation is the model for a thin extensible film under the influence of gravity and surface tension and $\gamma\not=0$ is a physical constant. Equation \eqref{eqdiv} appeared in the classical article of Serrin  (\cite[p. 477--478]{se}) and studied later in the context of the maximum principle (\cite{be,pp}). Therefore, we may simply view the study of $\phi$-minimal surfaces as a problem of a particular prescribed mean curvature equation, namely, \eqref{hg} or \eqref{hg2}.

In the class of  $\phi$-minimal surfaces, it is interesting to know explicit examples because a number of such examples makes richer this family of surfaces. Comparing with the theory of minimal surfaces ($H=0$) in $\r^3$, it is natural to  impose some geometric property to the surface that makes easier the study of \eqref{hg}, such as that the surface is rotational, helicoidal, ruled or translation. This situation has been studied in  \cite{ha,mh}. Following the above two motivations, in this paper we consider $\phi$-minimal surfaces of Riemann type and of translation type and that we now explain.

Firstly,  we consider surfaces which are constructed by a $1$-parameter (smooth) family of circles, not all necessarily with the same radius. Following Enneper (\cite{en1}), we call a such surface a {\it cyclic surface}. So, surfaces of revolution are examples of cyclic surfaces. The classical theory of minimal surfaces asserts that besides the plane, the only examples of cyclic minimal surfaces in $\r^3$ is the catenoid, which it is the only rotational minimal surface, or the surface belongs to a family of minimal surfaces discovered by Riemann \cite{ri}. Each Riemann minimal example is constructed by circles and all them lie  in parallel planes. We refer the Nitsche' book for a general reference \cite{ni}. In general, we say that a \emph{surface of Riemann type} is a cyclic surface such that the circles of the foliation lie in parallel planes.  In the class of $\phi$-minimal surfaces with a log-linear density, we prove that the only cyclic surfaces must be surfaces of revolution. Exactly, we show:

\begin{theorem}\label{t1}
Let $S\subset\r^3$ be a$\phi$-minimal cyclic surface for a log-linear density. Then the planes of the foliation are parallel.
\end{theorem}

\begin{theorem}\label{t2} Let $S\subset\r^3$ be a surface of Riemann type foliated by circles in parallel planes, all them orthogonal to a vector $\vec{v}\in\r^3$. If  $S$ is a $\phi$-minimal surface for a log-linear density, then the vector $\vec{v}$ is proportional to the $\phi$-vector density and $S$ is a surface of revolution whose axis is parallel to $\vec{v}$.
\end{theorem}

 If we allow that $\phi$   can constantly vanish, then the above two results  summarize as follows:

\begin{corollary} Let  $e^\phi$ be a  density in $\r^3$, where $\phi(x,y,z)=\alpha x+\beta y+\gamma z$. Then the existence of  non-rotational cyclic surfaces occurs if and only if $\phi=0$ and in such a case,  $S$ is a Riemann minimal classical example.
\end{corollary}

The second setting of examples of $\phi$-minimal surfaces appears when we study Eq. \eqref{eqdiv} by separation of variables $z=f(x)+g(y)$. A surface $S$ which is a graph of such a function is called a \emph{translation surface}. Let us observe that $S$ is the sum of two planar curves, namely, $x\mapsto (x,0,f(x))$ and $y\mapsto (0,y,g(y))$. Then $S$  has the property that the translations of a parametric curve $x=ct$ by the parametric curves $y=ct$ remain in $S$ (similarly for the parametric curves $x=ct$). For  minimal surfaces in $\r^3$, the only example, besides the plane, is the  Scherk's minimal surface $z=\log(\cos (\lambda y))-\log(\cos(\lambda x))$, $\lambda>0$ (\cite{ni}). Here we study translation surfaces that are $\phi$-minimal  with a log-linear density. In \cite{mh} it has been considered the case $\alpha=\beta=0$, proving that the only examples occur when $f$ or $g$ is linear, that is, $S$ is a cylindrical surface whose rulings are parallel to the $xz$-plane or to the $yz$-plane. We extend this result assuming $\phi$ in all its generality:

\begin{theorem}\label{t3} Let $S\subset\r^3$ be a translation surface $z=f(x)+g(y)$. If $S$ is a $\phi$-minimal surface for a log-linear density $\phi(x,y,z)=\alpha x+\beta y+\gamma z$, then $f$ or $g$ is a linera function and the surface is cylindrical whose rulings are parallel to the $xz$-plane or to the $yz$-plane.
\end{theorem}

In fact, we obtain new examples of $\phi$-minimal surfaces of translation type for values $\alpha,\beta\not=0$.

%%%%%%%%%%%
\section{Preliminaries}
%%%%%%%%%%%%%
Consider $S$ a surface in $\r^3$ with a  log-linear density $e^\phi$, where $\phi(x,y,z)=\alpha x+\beta y+\gamma z$, $\alpha,\beta,\gamma$ not all zero. Because all our results are local, we suppose that $S$ is oriented and denoted by $N$ its Gauss map.  The expression of $d\phi/dN$ is then
$$\frac{d\phi}{dN}=\alpha N_1+\beta N_2+\gamma N_3,$$
 where $N=(N_1,N_2,N_3)$ are the  coordinates of $N$ with respect to the canonical basis of $\r^3$.
From \eqref{hg2}, the $\phi$-minimality condition $H_\phi=0$ expresses as
\begin{equation}\label{HN}
H=\frac12(\alpha N_1+\beta N_2+\gamma N_3).
\end{equation}

 We now compute $H$ when $S$ is a surface of Riemann type. Without loss of generality, we suppose  that the surface   is foliated by circles contained in planes parallel to the $xy$-plane. Then the surface parametrizes locally as
 $$X(s,t)=(a(s),b(s),s)+r(s)(\cos(t),\sin(t),0),\ s\in I, t\in \r,$$
 where $a,b,r$, $r>0$ are smooth functions defined in some interval $I$. In fact, the results that we will obtain hold for surfaces that are foliated by pieces of circles because it is enough that the range of $t$ is an interval of $\r$. The computation of the Euclidean mean curvature $H$ yields
 $$H=\frac{1+(a'+r'\cos(t))^2+(b'+r'\sin(t))^2-r(r''+a''\cos(t)+b''\sin(t))}{2r(1+(r'+a'\cos(t)+b'\sin(t))^2)^{3/2}}.$$
 Here $'$ denotes the derivative with respect to $s$ and we also drop the dependence of the functions $a,b$ and $r$ on the variable $s$.
 The mean curvature function $H$ is computed with respect to Gauss map
 $$N=\frac{X_s\times X_t}{|X_s\times X_t|}=\frac{(-\cos(t),-\sin(t),r'+a'\cos(t)+b'\sin(t))}{\sqrt{1+(r'+a'\cos(t)+b'\sin(t))^2}}.$$

 \begin{remark}\label{rem}We point out that after a change of coordinates, the density function $\phi$ can  prescribe to be linear in one variable, as for example, $\phi(x,y,z)=z$. But in such a case, the planes of the foliation containing the circles of $S$ change by the above change of coordinates. Thus if we have assumed that the planes are parallel to a given plane, in our case, the plane of equation $z=0$, then the function $\phi$ has to be assumed in all its generality.
  \end{remark}

  In order to simplify the notation, let
  $$U=1+(a'+r'\cos(t))^2+(b'+r'\sin(t))^2-r(r''+a''\cos(t)+b''\sin(t)).$$
  $$W=1+(r'+a'\cos(t)+b'\sin(t))^2.$$
  Moreover,
  $$N_1=-\frac{\cos(t)}{\sqrt{W}},\ N_2=-\frac{\sin(t)}{\sqrt{W}},\ N_3=\frac{r'+a'\cos(t)+b'\sin(t)}{\sqrt{W}}.$$

  %%%%%%%%%%%%%%%
  \section{Proof of Theorem \ref{t2}}
  %%%%%%%%%%%%%%%%%%%%
Let $S$ be a  $\phi$-minimal surface of Riemann type with a log-linear density. Then \eqref{HN}  writes as
\begin{equation}\label{ww}
U-rW\left(-\alpha \cos(t)-\beta\sin(t)+\gamma(r'+a'\cos(t)+b'\sin(t))\right)=0.
\end{equation}
In order to prove Th. \ref{t2}, we have to show two things. First, that $\alpha=\beta=0$ and second, that the curve of centers $s\longmapsto (a(s),b(s),s)$ is parallel to the $z$-axis, that is, that $a$ and $b$ are constant functions, or equivalently, $a'=b'=0$.

  After a straightforward computation, Eq. \eqref{ww} can viewed as a  polynomial equation
\begin{equation}\label{fou}
\sum_{n=0}^3 \left(A_n(s)\cos(nt)+B_n(s)\sin(nt)\right)=0.
\end{equation}
 Because the trigonometric functions $\{\cos(nt),\sin(nt):n\in{\mathbb Z}\}$ are independent linearly, then the coefficient functions vanish identically, that is, $A_n=B_n=0$ for $0\leq n\leq 3$. The work to do is firstly the computation of the coefficients of the greatest degree in \eqref{fou}. As we go solving the corresponding equations $A_n=0$ and $B_n=0$, and thanks to this new information, we come back to \eqref{fou} to compute the next coefficients of lower degree until we obtain the desired result. We point out that the use of a symbolic program (as Mathematica or Maple)  reduces meaningfully the computations.

 We distinguish cases according the values of $\alpha,\beta$ and $\gamma$.

 %%%%%%%%%%%%%%%%%%%%%%%%
 \subsection{Cases where  some $\alpha,\beta,\gamma$ is $0$}
 %%%%%%

 First we suppose that two of the three constants are $0$.

 \begin{enumerate}
 \item Case $\beta=\gamma=0$. The computations of $A_3$ and $B_3$ gives
 $$A_3=\frac14\alpha  r(a'^2-b'^2),\ \ B_3=\frac12\alpha r a'b'.$$
 We  deduce that $a'=b'=0$ in the interval $I$, that is, $a$ and $b$ are constant functions.   Now \eqref{fou} is a polynomial equation  of degree $n=1$, with $A_1=\alpha r (1+r'^2)$, and $A_1=0$ yields a contradiction.
 \item Case $\alpha=\gamma=0$. This is similar to the previous case.
 \item Case $\alpha=\beta=0$. Then
 $$A_3=-\frac14\gamma ra'(a'^2-3b'^2),\ \ B_3=-\frac14\gamma r b'(3a'^2-b'^2).$$
 Again we deduce $a'=b'=0$ and the functions $a$ and $b$ are constant.  Now the degree of  polynomial equation \eqref{fou} is simply $n=0$, obtaining
 $$r(r''+\gamma r'(1+r'^2))-(1+r'^2)=0.$$
 \end{enumerate}

 Now we study the case that only  one of the constants  $\alpha,\beta$ or $\gamma$ is $0$.
 \begin{enumerate}
 \item Case $\alpha=0$ and $\beta,\gamma\not=0$. We have
 $$A_3=-\frac14ra'\left(\gamma a'^2+b'(2\beta-3\gamma b')\right)$$
 $$B_3=\frac14 r\left( a'^2(\beta-3\gamma b')+b'^2(\gamma b'-\beta)\right).$$
 From $A_3=0$, we distinguish  two cases.
\begin{enumerate}
 \item Assume $a'=0$. Then $B_3$ reduces $B_3=rb'^2(\gamma b'-\beta)/4$. We have two possibilities. If $b'=0$, then $b$ is a constant function, but now  the coefficient $B_1$ is $B_1=\beta r(1+r'^2)$, a contradiction. Thus $b'\not=0$. Now $B_3=0$ writes as $\gamma b'-\beta=0$. From $A_2=0$ we deduce $ r'(3\gamma b'-2\beta)=0$.
   If $r'\not=0$, we obtain a contradiction.  Thus $r'=0$  and $b(s)=\beta s/\gamma+b_0$, $b_0\in\r$. But now \eqref{fou} reduces into $1+\beta^2/\gamma^2=0$, a contradiction.
\item Assume $a'\not=0$. Then $\beta A_3-\gamma B_3=0$ writes as
$$b'(2\beta^2+3\gamma^2a'^2-2\beta\gamma b'-\gamma^2 b'^2)=0.$$
Suppose $b'=0$. Then $A_3=-\gamma ra'/4$, a contradiction.  Thus $b'\not=0$. Then $A_3=0$ and the above equation write now, respectively, as
$$\gamma a'^2+2\beta b'-3\gamma b'^2=0$$
$$\ 3\gamma^2 a'^2-2\beta\gamma b'-\gamma^2b'^2+2\beta^2=0.$$
Multiplying the first equation by $\gamma$ and subtracting the second one,  we get $ \gamma^2 a'^2+(\beta-\gamma b')^2=0$, which implies $\gamma=a'=0$, a contradiction.
\end{enumerate}

 \item Case $\beta=0$  and $\alpha,\gamma\not=0$. It is similar to the above case.
 \item Case $\gamma=0$ and $\alpha,\beta\not=0$. The computation of \eqref{fou} yields
 $$A_3=\frac14 r(\alpha a'^2-2\beta a'b'-\alpha b'^2),\ B_3=\frac14r(\beta a'^2+2\alpha a'b'-\beta b'^2)$$
 Combining $A_3=B_3=0$, we have
 $$(\alpha^2+\beta^2)(a'^2-b'^2)=0,\ (\alpha^2+\beta^2)a'b'=0.$$
 Thus $a'=b'=0$ and $a$ and $b$ are constant functions.
Now \eqref{fou} simplifies, obtaining $A_1=r\alpha$ and $B_1=r\beta$, a contradiction with the fact that $\alpha,\beta\not=0$.
 \end{enumerate}

\subsection{Case  $\alpha \beta\gamma\not=0$}
The computation of \eqref{fou} together $A_3=B_3=0$ implies
\begin{equation}\label{g1}
\alpha a'^2-\gamma a'^3-2\beta a'b'-\alpha b'^2+3\gamma a'b'^2=0
\end{equation}
\begin{equation}\label{g2}
\beta a'^2+2\alpha a'b'-3\gamma a'^2 b'-\beta b'^2+\gamma b'^3=0.
\end{equation}
Multiplying the first equation by $b'$ and subtracting the second one multiplied by $a'$, we get
$(a'^2+b'^2)(\beta a'+\alpha b'-2\gamma a'b')=0$. If $a'^2+b'^2=0$, then $a$ and $b$ are constant functions and this gives immediately $A_1=\alpha r(1+r'^2)$, a contradiction. Thus suppose that $a'$ and $b'$ do not vanish simultaneously and so,
\begin{equation}\label{ab}
\beta a'+\alpha b'-2\gamma a'b'=0.
\end{equation}
Using the above expression, the equations \eqref{g1} and \eqref{g2} simplify, respectively, as
$$a'(\alpha a'-\gamma a'^2-\beta b'+\gamma b'^2)=0$$
$$b'(\alpha a'-\gamma a'^2-\beta b'+\gamma b'^2)=0.$$
\begin{enumerate}
\item Case $a'=0$. It follows that $A_3=-\alpha rb'^2/2$, obtaining  $b'=0$, a contradiction.
\item Case $b'=0$. It is analogous to the previous case $a'=0$.
\item The remaining case is that
 $$\alpha a'-\gamma a'^2-\beta b'+\gamma b'^2=0.$$
This equation together \eqref{ab} allows to solve in $a',b'$, obtaining
 $$a'=\frac{\alpha}{\gamma},\ \ b'=\frac{\beta}{\gamma}.$$
 Let $a(s)=\alpha a/\gamma+a_0$ and $b(s)=\beta s/\gamma+b_0$, $a_0,b_0\in\r$. Then
 $$A_2=\frac{(\beta^2-\alpha^2)rr'}{2\gamma},\ \ B_2=-\frac{\alpha\beta rr'}{\gamma}.$$
 Hence $r'=0$.  But now Eq. \eqref{fou} reduces into $1+(\alpha^2+\beta^2)/\gamma^2=0$, obtaining a contradiction.
 \end{enumerate}

Arrived at this step of the arguments, we summarize the results obtained up now. Consider $S$ a surface of Riemann type which is foliated by circles in parallel planes to the $xy$-plane. If $S$ is a $\phi$-minimal surface for a linear function $\phi$, then $\phi$ must be $\phi(x,y,z)=\gamma z$, that is, the density of Euclidean space is $e^{\gamma z}$. Under this condition, the functions $a$ and $b$ are constant and the curve of  the centers of the circles is a straight-line parallel to the $z$-axis. The surface parametrizes locally as
$$X(s,t)=(a,b,s)+r(s)(\cos(t),\sin(t),0),\ \ a,b\in\r, r(s)>0$$
where the function $r=r(s)$ satisfies
$$ r(r''+\gamma r'(1+r'^2))=1+r'^2.$$

 This is the equation of the rotational solutions of \eqref{eqdiv} and, in particular, the existence is assured at least locally.  We point out that  it has been  shown the existence of  convex, rotationally symmetric  translating solitons which are graphs on the $xy$-plane and asymptotic to a paraboloid \cite{aw}. Moreover the only entire convex solution to \eqref{eqdiv}  must be rotationally symmetric in an appropriate coordinate system \cite{wa}.

In a similar way, we   get a  result as in Th. \ref{t2} for constant $\phi$-mean curvature surfaces, that is, $H_\phi$ is constant on the surface. The arguments are equal except that the computation are longer. For surfaces of Riemann type and in the case that $\alpha=\beta=0$, we prove:
\begin{theorem}
 Let $S\subset\r^3$ be a surface foliated by circles contained in parallel planes to the $xy$-plane. Suppose  $H_\phi$ is a nonzero constant for a log-linear density $e^{ z}$. Then $S$ is a surface of revolution whose axis is parallel to the $z$-axis.
\end{theorem}
\begin{proof}

Suppose $H_\phi=c/2$, where $c$ is a nonzero number. Then we have
 $$H=\frac12( N_3+c).$$
By using the notation in the above section, this equation writes now as
$$ U-rW  (a'\cos(t)+b'\sin(t)+r')= c r W^{3/2}.$$
Squaring in this identity and placing all the terms in the left hand side, we obtain a polynomial equation
\begin{equation}\label{fouh}
\sum_{n=0}^6 \left(A_n(s)\cos(nt)+B_n(s)\sin(nt)\right)=0.
\end{equation}
Again, all coefficients $A_n=B_n=0$ must vanish identically for $0\leq n\leq 6$. From $A_6=0$ and $B_6=0$, we have
$$(c^2-1)(a'^6-15a'^4b'^2+15a'^2b'^4-b'^6)=0$$
$$(c^2-1)a'b'(3a'^4-10a'^2b'^2+3b'^4)=0.$$
\begin{enumerate}
 \item Case $c^2\not=1$. Then we deduce immediately $a'=0$ and $b'=0$.
\item Case $c^2=1$. Then   $A_4=0$ and $B_4=0$ write, respectively, as
$$(r+4r')(a'^4-6a'^2b'^2+b'^4)+ra''(6a'b'^2-2a'^3)+rb''(6a'^2b'-2b'^3)=0$$
$$2a'(r+4r')(-a'^2b'+b'^3)+ra''b'(3a'^2-b'^2)+ra'b''(a'^2-3b'^2)=0.$$
Hence we deduce
$$a''=\frac{ra'+4a'r'}{2r},\ b''=\frac{rb'+4b'r'}{2r}.$$
Using these values of $a''$ and $b''$, the computations of $A_3=0$ and $A_2=0$ imply that $a'=b'=0$, or
$$2+2a'^2+2b'^2+rr'+2r'^2-2rr''=0$$
$$8r'(1+a'^2+b'^2+r'^2)+r(1+4r'^2-8r'r'')=0,$$
respectively. In the latter case, multiplying the first equation by $-4r'$ and adding the second equation, we obtain $r=0$, a contradiction.
\end{enumerate}
In any of the two cases, we conclude $a'=b'=0$, which proves that $S$ is a surface of revolution. As the curve of centers is $s\mapsto (a,b,s)$, $a,b\in\r$, then the axis of rotation is parallel to the $z$-axis,  showing the result.
 \end{proof}

%%%%%%%%%%%%
\section{Proof of Theorem \ref{t1}}
%%%%%%%%%%%5

We follow the same ideas as in the proof of the result for minimal surfaces in $\r^3$, see \cite[p. 85-86]{ni}. The proof is by contradiction and we assume that the planes of the foliation are not parallel. Because the result is local, we only work  in an interval of the foliation that defines the surface.  Let $\{E_1=(1,0,0), E_2=(0,0,1), E_3=(0,0,1)\}$ the canonical basis of $\r^3$. Because the planes of the foliation are not parallel, and after a change of coordinates,  we can suppose that the $\phi$-vector density is $\vec{v}=E_3$. Let $s$ be the parameter of the uniparametric family of circles that defines the surface, let $r(s)>0$ be the radius in each $s$-leaf. Consider $\Gamma=\Gamma(s)$ a curve parametrized by the arc-length which is orthogonal to each $s$-plane. This means that the vector $\Gamma'(s)$ is orthogonal to the $s$-plane. Because we suppose that the planes are not parallel, then $\Gamma$ is not a straight-line. Let $\{\t(s),\n(s),\b(s)\}$ be the Frenet frame of $\Gamma$, where $\t, \n$ and $\b$ are the tangent vector, normal vector and binormal vector of $\Gamma$, respectively. If $\c=\c(s)$ is the curve of centers of the circles of the foliation, then $S$ parametrizes locally  as
$$X(s,t)=\c(s)+r(s)\left(\cos(t)\n(s)+\sin(t)\b(s)\right).$$
Let $\c'(s)=u(s)\t(s)+v(s)\n(s)+w(s)\b(s)$, where $u,v,w$ are smooth functions. Consider the Frenet equations of $\Gamma$:
$$\begin{array}{ccc}
\t'&=&\kappa\n\\
\n'&=&-\kappa\t+\sigma\b\\
\b'&=&-\sigma\n
\end{array}$$
Here $'$ denotes the derivative with respect to the $s$-parameter and $\kappa$ and $\sigma$ are the curvature and torsion of $\Gamma$, respectively. Since $\Gamma$ is not a (piece of) straight-line, then $\kappa\not=0$. Now the $\phi$-minimal condition writes simply as $2H=\langle N,E_3\rangle=N_3$, where $N$ is the Gauss map of $S$. Now we compute each one the terms of  this equation. The Gauss map is
$$N=\frac{(\cos(t) P+\sin(t)Q)\t-\cos(t)M \n-\sin(t)M\b}{\sqrt{M^2+(\cos(t)P+\sin(t)Q)^2}},$$
where
$$M=u-r\kappa\cos(t),\ P=v+r'\cos(t)-r\tau\sin(t),\ Q=w+r'\sin(t)+r\tau\cos(t).$$
Then
$$N_3=\frac{(\cos(t)P+\sin(t)Q)t_3-\cos(t)M n_3-\sin(t)M b_3}{\sqrt{M^2+(\cos(t)P+\sin(t)Q)^2}},$$
where $t_3,n_3$ and $b_3$ are the third coordinates of $\t,\n$ and $\b$ with respect to the canonical basis of $\r^3$.
For the computation of the mean curvature $H$,  we have to differentiate twice the parametrization $X(s,t)$. In all these computations, we use the Frenet equation and the fact that $\{\t,\n,\b\}$ is an orthonormal basis of $\r^3$, with $\mbox{det}(\t(s),\n(s),\b(s))=1$ for all $s$. After a straightforward computation, the condition $2H-N_3=0$ is expressed as a trigonometric polynomial on $\{\cos(nt),\sin(nt): n\in{\mathbb Z}\}$, namely,
$$\sum_{n=0}^4\left(A_n(s)\cos(nt)+B_n(s)\sin(nt)\right)=0,$$
where $A_n$ and $B_n$ are smooth functions. Thus all coefficients $A_n$ and $N_n$ vanish in the $s$-interval.

The leader coefficients (for $n=4)$ are
$$A_4=\frac18\kappa r^2\left(n_3(\kappa^2r^2+v^2)+2b_3vw+n_3w^2\right)$$
$$B_4=-\frac18\kappa r^2\left(2n_3vw-b_3(\kappa^2+r^2+v^2-w^2)\right).$$
The linear combination $b_3 A_4-n_3 B_4=0$ simplifies into
\begin{equation}\label{b3}
(b_3^2+n_3^2)vw=0.
\end{equation}
Let us observe that if $b_3=n_3=0$, then the vectors $\n(s)$ and $\b(s)$ are orthogonal to $E_3$ for all $s$. In such a case, $\t(s)=\pm E_3$ for all $s$ and it follows that $\Gamma'(s)$ would be parallel to the vector $E_3$, in particular, $\Gamma$ is a (vertical) straight-line, a contradiction. Thus $n_3$ and $b_3$ can not vanish mutually. We discuss Eq. \eqref{b3} case-by-case.
\begin{enumerate}
\item Case $n_3b_3\not=0$. Then $v=0$ or $w=0$.
\begin{enumerate}
\item Sub-case $v=0$. Then $A_4=0$ gives $w^2=\kappa r^2$ and it follows $w=\pm\kappa r$. Suppose $w=\kappa r$ (similar if $w=-\kappa r)$. Then
    $$A_3=\frac12\kappa^2r^3(n_3u+b_3r'),\ \ B_3=\frac12\kappa^2r^3(b_3u-n_3r').$$
This implies $u=0$ and $r'=0$. Finally for $n=2$, we have
$$A_2=-\frac12r^4\kappa^3n_3,\ \ B_2=-\frac12 r^4 \kappa^3 b_3,$$
obtaining $n_3=b_3=0$, a contradiction.
\item Sub-case $w=0$. Now
$$A_4=-\frac18 \kappa r^2 n_3(\kappa^2 r^2+v^2),$$
and $A_4=0$ implies $n_3=0$, a contradiction.
\end{enumerate}
\item  Case $n_3=0$ and $b_3\not=0$. Then $A_4=0$ reduces into $b_3vw=0$. Thus $v=0$ or $w=0$. Suppose $v=0$ (similar if $w=0$). Then $w=\pm\kappa r$. Now $A_3=0$ and $B_3=0$ write as
     $b_3r'=0$ and $b_3u=0$, respectively. Since $b_3\not=0$, it follows $r'=0$ and $u=0$. Finally, the coefficient $B_2$ is $B_2=-r^4b_3\kappa^3/2$, and $B_2=0$ yields the desired contradiction.
\item Case $b_3=0$ and $n_3\not=0$. This   is similar to the  previous case.
\end{enumerate}

%%%%%%%%%%%%%%%%%%%%5
\section{Proof of Theorem \ref{t3}}
%%%%%%%%%%%%%%%%%%%%5

Let $S\subset\r^3$ be a translation surface $z=f(x)+g(y)$, where $f:I\subset\r$ and $g:J\subset\r\rightarrow\r$ are smooth functions defined in intervals of $\r$. Let $v=(\alpha,\beta,\gamma)\in\r^3$ be a non-zero vector. Suppose that $S$ is a $\phi$-minimal surface with $\vec{v}$ as $\phi$-vector density.  As it was pointed out in Remark \ref{rem}, we can not do a change of coordinates to prescribe the $\phi$-vector density $\vec{v}$ because in such a case, the surface is not of the form $z=f(x)+g(y)$. Recall that in \cite{mh}, it has been considered the case that $\alpha=\beta=0$, proving that the only possibility is that $f$ (or $g$) is linear.

Let $X(x,y)=(x,y,f(x)+g(y))$ be the parametrization of $S$. With respect to the unit normal vector field
$$N=\frac{X_x\times X_y}{|X_x\times X_y|}=\frac{(-f',-g',1)}{\sqrt{1+f'^2+g'^2}}$$
the mean curvature $H$ is
$$H=\frac{(1+g'^2)f''+(1+f'^2)g''}{2(1+f'^2+g'^2)^{3/2}}.$$
Thus $S$ is a $\phi$-minimal surface if and only if
$$\frac{(1+g'^2)f''+(1+f'^2)g''}{(1+f'^2+g'^2)^{3/2}}=\frac{-\alpha f'-\beta g'+\gamma}{\sqrt{1+f'^2+g'^2}},$$
or equivalently,
\begin{equation}\label{eqt1}
(1+g'^2)f''+(1+f'^2)g''=(1+f'^2+g'^2)(-\alpha f'-\beta g'+\gamma).
\end{equation}
We differentiate \eqref{eqt1} with respect to $x$ and next with respect to $y$, obtaining
\begin{equation}\label{eqt2}
g'g''f'''+f'f''g'''=-f''g''(\alpha g'+\beta f').
\end{equation}
For completeness, we consider the case $\alpha=\beta=0$, doing a different proof than  in \cite{mh}. In such a situation, $g'g''f'''+f'f''g'''=0$. If $f''=0$ (resp. $g''=0$), then $f$ (resp. $g$) is linear. Assume $f''g''\not=0$, in particular, $f'g'\not=0$. Then there exists $\lambda\in\r$ such that
$$\frac{f'''}{f'f''}=-\frac{g'''}{g'g''}=2\lambda.$$
A direct integration yields $f''=a_1+\lambda f'^2$ and $g''=b_1-\lambda g'^2$, $a_1,b_1\in\r$. Then \eqref{eqt1} simplifies now into
$$(a_1+b_1-\gamma)+(b_1-\gamma+\lambda)f'^2+(a_1-\gamma-\lambda)g'^2=0.$$
If we view this expression first as a polynomial equation on $f'$ and second as a polynomial equation on $g'$, we deduce $a_1+b_1-\gamma=0$, $a_1-\gamma-\lambda=0$ and $a_1+b_1-\gamma=0$, obtaining $\gamma=0$,  a contradiction.

Once proved the result for $\alpha=\beta=0$, we discuss the rest of cases. First, we assume $f''g''\not=0$ and we will arrive to a contradiction. Since the case $\alpha\not=0$ and $\beta=0$ is similar to $\alpha=0$ and $\beta\not=0$, there are only two possibilities.

\subsection{Case $\alpha=0$,  $\beta\not=0$ }
Dividing \eqref{eqt2} by $f'f''g'g''$, we have
$$\frac{f'''}{f'f''}=-\frac{\beta}{g'}-\frac{g'''}{g'g''}.$$
Because the left hand side depends on $x$ and the right hand side depends on $y$, there  exists $\lambda\in\r$ such that
$$\frac{f'''}{f'f''}=-\frac{\beta}{g'}-\frac{g'''}{g'g''}=2\lambda.$$
Then $f''=\lambda f'^2+a_1$, $a_1\in\r$.  If we view  \eqref{eqt1} as a polynomial on $f'$ and since $f'\not=0$, then all coefficients vanish, obtaining the next two equations:
$$\lambda(1+g'^2)+g''+\beta g'-\gamma=0,\ \ (1+g'^2)(a_1+\beta g'-\gamma)=0.$$
In particular, $a_1+\beta g'-\gamma=0$. As $\beta\not=0$, we conclude  $g''=0$, a contradiction.

\subsection{Case $\alpha,\beta\not=0$}
 From \eqref{eqt2}, we have
$$\frac{f'''}{f'f''}+\frac{\alpha}{f'}=-\frac{\beta}{g'}-\frac{g'''}{g'g''}.$$
Again there exists a constant $\lambda\in\r$ such that
$$\frac{f'''}{f'f''}+\frac{\alpha}{f'}=-\frac{\beta}{g'}-\frac{g'''}{g'g''}=2\lambda.$$
Then a first integration implies
$$f''=-\alpha f'+\lambda f'^2+a_1,\ a_1\in\r.$$
Substituting into \eqref{eqt1}, we get a polynomial on $f'(x)$ of type
$$\sum_{n=0}^3 B_n(y) f'(x)^n=0.$$
Because $f'\not=0$ identically, then the coefficients $B_n$ must vanish. However a straightforward computation of $B_3$ leads to $B_3=\alpha$, obtaining a contradiction.

Once that the case $f''g''\not=0$ has been discarded, then we conclude $f$ or $g$ is linear. This finishes the proof of Th. \ref{t3}.

In order to summarize the results, and by the symmetry of the roles of $f$ and $g$, we will suppose that $f''=0$.

\begin{theorem} Let $S\subset\r^3$ be a translation surface $z=f(x)+g(y)$. Suppose $S$ is a $\phi$-minimal surface for a log-linear density $\phi(x,y,z)=\alpha x+\beta y+\gamma z$. Then:
\begin{enumerate}
\item There exists $a_1,a_0\in\r$ such that $f(x)=a_1x+a_0$.
\item The surface is cylindrical and the rulings are parallel to the $xz$-plane.
\item The function $g$ satisfies
\begin{equation}\label{g}
(1+a_1^2)g''=(1+a_1^2+g'^2)(-\beta g'-\alpha a_1+\gamma).
\end{equation}
\end{enumerate}
\end{theorem}
We remark that \eqref{g} is an ODE and thus the existence is assured at least locally, proving the existence of $\phi$-minimal surfaces of translation type.

To finish this section, we give some particular solutions of \eqref{g}. Let   $f(x)=a_1x+a_0$.
\begin{enumerate}
\item Consider $a_1=0$, $\gamma=0$ and $g(y)=b_2$, $b_2\in\r$. The surface is the horizontal plane of equation $z=a_2+b_2$.
    \item  Let  $g(y)=b_1 y+b_0$, $b_1,b_0\in\r$ and  $\gamma=\alpha a_1+\beta b_1$. The surface is a plane again.
\item If $\alpha=\beta=0$, we have (see\cite{mh}):
$$g(y)=b_2-\frac{1+a_1^2}{\gamma}\log\left[\cos\left(\frac{\gamma y+b_1}{\sqrt{1+a_1^2}}\right)\right],\ \ b_1,b_2\in\r.$$
\item  If $\gamma-\alpha a_1=0$, then
$$g(y)=b_1-\frac{\sqrt{1+a_1^2}}{\beta}\mbox{ arc}\sin\left(e^{b_2-\beta y}\right),\ \ b_1,b_2\in\r.$$

\end{enumerate}

\begin{remark} We point out that if $f(x)=a_0$ and $\alpha=\beta=0$, then we have the one-dimensional version of \eqref{eqdiv}, which is expressed as
$$\frac{y''}{1+y'^2}=\gamma,$$
where $y=y(x)$. This equation says that the curvature of the planar curve $y(x)$  is proportional to the $y$-coordinate of its unit normal vector.
\end{remark}
%%%%%%%%%%%%%%%%%%%

\end{document}